\documentclass[11pt]{article}
\UseRawInputEncoding
\usepackage{amssymb,tikz}
\usepackage{color}
\usepackage{soul}
\usepackage{mathtools}

\newcommand{\N}{\mathbb{N}}
\newcommand{\R}{\mathbb{R}}

\newcommand{\ord}{\mbox{\rm ord}}
\newcommand{\ini}{\mbox{\rm in}}
\newcommand{\Ini}{\mbox{\sc{I}\rm n}}
\newcommand{\bideg}{\mbox{\rm bideg}}

\newcommand{\supp}{\mathrm{supp}\,}

\newlength{\szer}

\newtheorem{defi}{Definition}[section]
\newtheorem{nota}[defi]{Remark}

\newtheorem{teorema}[defi]{Theorem}
\newtheorem{prop}[defi]{Proposition}
\newtheorem{lema}[defi]{Lemma}
\newtheorem{coro}[defi]{Corollary}

\newenvironment{proof}[1][Proof]{\textbf{#1.} }{\
\rule{0.5em}{0.5em}}

\begin{document}
\title{A note on the plane curve singularities in positive characteristic
\footnotetext{
     \noindent   \begin{minipage}[t]{4in}
       {\small
       2010 {\it Mathematics Subject Classification:\/} Primary 14H20;
       Secondary 32S05.\\
       Key words and phrases:  Milnor number, Newton polygon, non-degeneracy.\\
       The first-named author was partially supported by the Spanish grant
       PID2019-105896GB-I00 funded by MCIN/AEI/10.13039/501100011033.}
       \end{minipage}}}

\author{Evelia R.\ Garc\'{\i}a Barroso and Arkadiusz P\l oski}

\maketitle

\begin{abstract} Given an algebroid plane curve $f=0$ over  an algebraically closed field  of characteristic $p\geq 0$ we consider the Milnor number $\mu(f)$,  the delta invariant $\delta(f)$á and the number $r(f)$ of its irreducible components. Put $\bar \mu(f)=2\delta(f)-r(f)+1$. If $p=0$ then $\bar \mu (f)=\mu(f)$ (the Milnor formula). If $p>0$ then $\mu(f)$ is not an invariant and $\bar \mu(f)$ plays the role of $\mu(f)$. Let $\mathcal N_f$ be the Newton polygon of $f$. We define the numbers $\mu(\mathcal N_{f})$ and $r(\mathcal N_{f})$ which can be computed by explicit formulas.
The aim of this note is to give a simple proof of the inequality
$\bar \mu(f)-\mu(\mathcal N_{f})\geq r(\mathcal N_{f})- r(f)\geq 0$ due to Boubakri, Greuel and Markwig.
We also prove that $\bar \mu(f)=\mu(\mathcal N_{f})$ when $f$ is non-degenerate.
\end{abstract}

\section*{Introduction}
The main objective of this note is to give a new and simple proof of \cite[Proposition 7]{BGreuel}.The paper is organised as follows. The Section \ref{requisitos} is a survey of prerequisites from the theory of algebroid curves (see \cite{Pfister-Ploski}). We define the invariant $\bar \mu$ which plays the crucial role in this note. In Section \ref{Newtonp} we use the Newton polygons that are vital to the proof of the main result. In Section \ref{sec:main} we make use of the Newton polygon $\mathcal N_f$, associated with a formal power series $f\in K[[x,y]]$ ($K$ is an algebraically closed field of  arbitrary characteristic), to compute the invariant $\bar \mu (f)$ and the number $r(f)$ of irreducible components of the curve $f(x,y)=0$. We define the numbers $\mu(\mathcal N_{f})$ and 
$r(\mathcal N_{f})$ which are combinatorial counterparts of $\bar \mu (f)$  and $r(f)$. Suppose that $f$ is a reduced power series. We give a new proof of \cite[Proposition 7]{BGreuel} which states
\begin{equation}\label{eqBG}
\bar \mu(f)-\mu(\mathcal N_{f})\geq r(\mathcal N_{f})- r(f)\geq 0.
\end{equation}

On the other hand, under the assumption of non-degeneracy introduced by Beelen and Pellikaan
\cite[Definition 3.14]{Be-Pe} we prove that
\begin{equation}
\label{eqBP}
\bar \mu(f)=\mu(\mathcal N_{f}) \;\;\; \hbox{\rm and }\;\;\; 
r(f)=r(\mathcal N_{f}).
\end{equation}

The inequality \eqref{eqBG} generalizes \cite[Theorem 1.2]{Ploski} where  \eqref{eqBG} is proved when the characteristic of the field $K$ is zero and $f$ is convenient. Section \ref{proof} is devoted to the proofs of \eqref{eqBG} and \eqref{eqBP}.

\section{Prerequisites}
\label{requisitos}

Let $K$ be an algebraically closed field of arbitrary characteristic. Let $f\in K[[x,y]]$ be a non-zero power series without constant term. The power series $f$ is reduced if it has not multiple factors.

In what follows we consider the equisingularity invariants of a reduced plane curve $\{f(x,y)=0\}$ (see \cite{Pfister-Ploski}): $r(f)$ is the number of irreducible factors of $f$, $c(f)=\dim_{K} \overline{\cal O}_f/{\cal C}$ is the {\it degree of the conductor}, where ${\cal O}_f=K[[x,y]]/(f)$, $ \overline{{\cal O}_f}$  is the integral closure of  ${\cal O}_f$ in the total quotient ring of ${\cal O}_f$ and ${\cal C}$ is the \emph{conductor} of ${\cal O}_f$, that is the largest ideal in ${\cal O}_f$  which remains an ideal in $ \overline{{\cal O}_f}$. Finally the {\it delta invariant} of $f$ is 
$\delta (f)=\dim_{K} \overline{\cal O}_f/{\cal O}_f$.  Since ${\cal O}_f$ is Gorenstein we get $c(f)=2\delta(f)$. 

 If $f\in K[[x,y]]$ is irreducible then the {\it semigroup of values} of $f(x,y)$, denoted by $\Gamma(f)$, is  defined as the set of intersection multiplicities $i_0(f,g)=\dim_K K[[x, y]]/(f,g)$ where $g$ runs over all power series in $K[[x,y]]$
such that $g\not\equiv 0$ (mod $f$). This semigroup is numerical, that is $\N\backslash \Gamma(f)$ is a finite set. Denote by $c$ the {\it conductor} of  $\Gamma(f)$, that is, the smallest element of $\Gamma(f)$ such that $c+N\in \Gamma(f)$ for any nonnegative integer $N$. The semigroup $\Gamma(f)$ admits a minimal system of generators $v_0<v_1<\cdots<v_g$ such that $\gcd(v_0,\ldots,v_g)=1$. We write $\Gamma(f)=\langle v_0,\ldots,v_g\rangle$. Put $e_{i}:=\gcd(v_0,\ldots,v_i)$ for $0\leq i\leq h$ and $n_{i}=\frac{e_{i-1}}{e_{i}}$ for $1\leq i \leq g$. 

If $f$ is irreducible then $c(f)$ equals to the conductor $c$ of the semigroup $\Gamma(f)$. Consequently
$c(f)=\sum_{k=1}^{g}(n_{k}-1)v_{k}-v_{0}+1$.\\

Let $ \mu(f)$ be the {\em Milnor number} of $f$ defined as the codimension of the ideal generated by
$\frac{\partial f}{\partial x},\frac{\partial f}{\partial y},$ that is $\mu(f)=i_0\left(\frac{\partial f}{\partial x},\frac{\partial f}{\partial y}\right)$. The {\em invariant Milnor number } of $f$  is defined to be 
$\bar \mu(f)=2\delta(f)-r(f)+1=c(f)-r(f)+1$ (see\cite{GB-P2018}).  If $p=0$ then $\bar \mu (f)=\mu(f)$ (the Milnor formula). If $p>0$ $\mu(f)$ is not an invariant and $\bar \mu(f)$ plays the role of $\mu(f)$.
Melle and Wall \cite{Melle-Wall}, based on a result of Deligne \cite{Deligne}, proved that $\mu(f) \geq \bar \mu (f)$.\\

Any plane reduced curve $\{f(x,y)=0\}$ is called a {\it tame singularity} if $\mu(f)=\bar \mu (f)$. If the characteristic of $K$ is zero any singularity of plane reduced curve is tame.\\

\begin{prop} $\;$\label{prop:1}
\begin{enumerate}
\item For any unit $u\in K[[x,y]]$ we get $\bar \mu(uf)=\bar \mu(f)$.
\item For every reduced power series $f\in K[[x,y]]$  we have $\bar \mu(f)\geq 0$  and $\bar \mu(f)=0$ if and only if $\ord f=1$. 
\item Let $f=g_{1}\cdots g_{s}$ be a reduced power series where $g_{i}\in K[[x,y]]$ are pairwise coprime. Then 
\[
\bar \mu(f)+s-1=\sum_{i=1}^{s} \bar \mu(g_{i})+2\sum_{1\leq i<j\leq s}i_{0}(g_{i},g_{j}).
\]
\end{enumerate}
\end{prop}
\begin{proof}
See \cite[Proposition 1.2, Remark 2.2]{GB-P2015}.
\end{proof}

If the characteristic of $K$ is positive then, in general, we have $\mu(uf)\neq \mu(f)$ (see \cite[page 63]{BGreuel}).

\section{Newton polygons and plane curve singularities}
\label{Newtonp}

A segment $S\subset \mathbb R^2$ is a {\it Newton edge} if its vertices $(\alpha,\beta)$, $(\alpha',\beta')$ lie in
$\mathbb N^2$ and $\alpha<\alpha'$, $\beta'<\beta$. Put $\vert S\vert_1=\alpha'-\alpha$, $\vert S\vert_2=\beta-\beta'$, $r(S)=\gcd(\vert S\vert_1,\vert S\vert_2)$  If $S,T$ are two Newton edges we define $[S,T]:=\min \{\vert S\vert_1\vert T\vert_2,\vert S\vert_2\vert T\vert_1\}$. If   $\frac{\vert S\vert_1}{\vert S\vert_2}<\frac{\vert T\vert_1}{\vert T\vert_2}$ then $[S,T]=\vert S\vert_1\vert T\vert_2$ (see Figure \ref{fig:[S,T]}).\\

\begin{figure}[h!] 
\begin{center}
\begin{tikzpicture}[x=0.5cm,y=0.5cm] 
\tikzstyle{every node}=[font=\small]
\fill[fill=yellow!40!white] (0,2) --(0,0)-- (3,0) --(3,2)-- cycle;
 \node at (1.5,1) {$\vert S\vert_1\vert T\vert_2$}; 
\node[draw,circle,inner sep=1.4pt,fill, color=black] at (0,4){};
\node[draw,circle,inner sep=1.4pt,fill, color=black] at (10,0){};
\node[draw,circle,inner sep=1.4pt,fill, color=black] at (3,2){};
\node [above] at (2,3) {$S$};
\node [above] at (7,1) {$T$};
\draw[->] (0,0) -- (11,0) node[right,below] {$\alpha$};
\draw[->] (0,0) -- (0,5) node[above,left] {$\beta$};
\draw[-, line width=0.5mm] (0,4) -- (3,2);
\draw[-, line width=0.5mm] (3,2) -- (10,0);
\draw[<->] (0,-0.8) to [bend right] (3,-0.8);
\draw[-] (1.5,-1.3) node[below] {$\vert S\vert_1$};
\draw[<->] (3.1,-0.8) to [bend right] (10,-0.8);
\draw[-] (6.5,-1.7) node[below] {$\vert T\vert_1$};
\draw[<->] (-0.5,4) to [bend right] (-0.5,2);
\draw[-] (-1.8,3.5) node[right,below] {$\vert S\vert_2$};
\draw[<->] (-0.5,1.9) to [bend right] (-0.5,0);
\draw[-] (-1.8,1.5) node[right,below] {$\vert T\vert_2$};
\end{tikzpicture}
\end{center}
\caption{$[S,T]=\vert S\vert_1\vert T\vert_2$}  
   \label{fig:[S,T]}
    \end{figure}

Let $K$ be an algebraically closed field of characteristic $p\geq 0$. Consider $f\in K[[x,y]]$  a  nonzero power series without constant term. Write $f=\sum_{\alpha,\beta}c_{\alpha\beta}x^{\alpha}y^{\beta}$. The {\it support} of $f$ is $\supp f=\{(\alpha,\beta)\in \N^2\;:\; c_{\alpha\beta}\neq 0\}$. The {\it Newton diagram} $\Delta(f)$ of $f$ is the convex hull
of  $\supp f+(\R_{\geq 0})^2$. The {\it Newton polygon} $\mathcal N_f$ of $f$ is  the set of compact faces of the boundary of $\Delta(f)$. We put $\vert \mathcal N_f \vert_1=\sum_{S\in \mathcal N_f}\vert S\vert_1$, $\vert \mathcal N_f \vert_2=\sum_{S\in \mathcal N_f}\vert S\vert_2$, $[\mathcal N_f,\mathcal N_f]=\sum_{S,T\in \mathcal N_f}[S,T]$ and $r(\mathcal N_f)=\sum_{S\in \mathcal N_f}r(S)+k+l$, where $k,l$ are maximal such that $x^{k}y^{l}$ divides $f$.\\

A power series $f\in K[[x,y]]$ is {\it convenient} if $f(x,0)f(0,y)\neq 0$; otherwise we will say that $f$ is {\em non-convenient}. When $f$ is convenient the curve $f(x,y)=0$ does not contain 
the axes. Hence there is  an edge $F\in \mathcal N_f$ with the vertex $(m,0)$, where $m=\ord f(x,0)$  and there is 
$E\in \mathcal N_f$ with the vertex $(0,n)$ where $n=\ord f(0,y)$. The edges $F$ and $E$ are not necessarily different.\\

Let $f(x,y)=\sum_{\alpha, \beta}c_{\alpha\beta}x^{\alpha}y^{\beta}\in K[[x,y]]$. Recall that the {\it order} of $f$ is $\ord f=\min\{\alpha+\beta\;:\;c_{\alpha\beta}\neq 0\}$ and the {\it initial part} of $f$ is $\ini f=\sum_{\alpha+ \beta=\ord f}c_{\alpha\beta}x^{\alpha}y^{\beta}$. \\

For any segment $S\in \mathcal N_f$ we put $\ini(f,S)=\sum_{(\alpha, \beta)\in S}c_{\alpha\beta}x^{\alpha}y^{\beta}$. Let $x^{\alpha(S)}y^{\beta(S)}$ be the monomial of highest degree dividing $\ini(f,S)$. Then $\ini(f,S)=x^{\alpha(S)}y^{\beta(S)}\overline \ini(f,S)$ where $\overline \ini(f,S)$ is a convenient power series. We say that $f$ is {\it non-degenerate} if $\overline \ini(f,S)$ is reduced for every $S\in \mathcal N_f$, that is it does not have multiple factors.

\begin{nota}
A power series $f$ is non-degenerate if and only if for any segment $S\in \mathcal N_f$
 the solutions of the system 
\[\left\{\begin{array}{l}
\frac{\partial }{\partial x}\ini(f,S)=0\\
\frac{\partial }{\partial y}\ini(f,S)=0\\
\ini(f,S)=0
\end{array}
\right.
\]
are contained in $\{xy=0\}$ (see \cite[Proposition 3.5]{G-N}. On the other hand $f$ is non-degenerate in the strong sense (Kouchnirenko \cite{Kouchnirenko}) if the solutions of the system 
\[\left\{\begin{array}{l}
\frac{\partial }{\partial x}\ini(f,S)=0\\
\frac{\partial }{\partial y}\ini(f,S)=0\\
\end{array}
\right.
\]
are contained in $\{xy=0\}$ for any segment $S\in \mathcal N_f$. In zero characteristic both definitions are equivalent (see \cite[Remark 3.15]{Be-Pe}).
Nevertheless  if the characteristic of $K$ is $p>0$ then the power series $f(x,y)=x^{p}+y^{p+1}$  is non-degenerate but it is not non-degenerate in the strong sense.\\
\end{nota}

Assume that $f$ is a convenient power series. Recall that $m=\ord f(x,0)$ and $n=\ord f(0,y)$.
We put 

\begin{equation}\label{convenient}
\mu(\mathcal N_f)=[\mathcal N_f,\mathcal N_f]-\vert \mathcal N_f \vert_1-\vert \mathcal N_f \vert_2+1
\end{equation}

 and \[
 \delta(\mathcal N_f)=\frac{1}{2}\left(\mu(\mathcal N_f)+r(\mathcal N_f)-1\right).
 \]

Note that 

\begin{itemize}
\item $\mu(\mathcal N_f)=2 \hbox{\rm(area of the polygon bounded by $\mathcal N_f$ and the axes)}-n-m+1$, which is called the {\em Newton number} of $f$.
\item $r(\mathcal N_f)=\hbox{\rm(number of integer points on } \mathcal N_f)-1$, and 
\item $\delta(\mathcal N_f)=$ number of integer points lying below $\mathcal N_f$ but not on the axes. This is a consequence of Pick's formula.
\end{itemize}

If $f$ is a reduced power series (not necessarily convenient) then we define:

\begin{equation}\label{mueverybody}
\mu(\mathcal N_f)=\sup_{m\in \mathbb N}\{\mu(\mathcal N_{f_{m}})\;:\;f_{m}=f+x^{m}+y^{m}\}.
\end{equation}

\noindent Like in the case of convenient power series we put 
\begin{equation}\label{eq:delta}
 \delta(\mathcal N_f)=\frac{1}{2}\left(\mu(\mathcal N_f)+r(\mathcal N_f)-1\right)
 \end{equation}
for any reduced power series.

\noindent Observe that if $f$ is convenient then the two definitions of $\mu(\mathcal N_f)$, \eqref{convenient} and \eqref{mueverybody}, coincide.

\noindent Let $f\in K[[x,y]]$ be a reduced power series and let $x^{d_{1}}y^{d_{2}}$ be the monomial of highest degree dividing $f$. We have $f=x^{d_{1}}y^{d_{2}}g$ where $g\in K[[x,y]]$ is a convenient power series or a unit. Since $f$ is reduced $d_{1}, d_{2}\leq 1$ and $(d_{1},d_{2})=(0,0)$ if and only if $f$ is convenient. We have $[\mathcal N_f,\mathcal N_f]=2(\hbox{\rm the area between  $\mathcal N_f$ and the lines $x-d_{1}=0$, $y-d_{2}=0)$.}$\\

\noindent The following nice formula is due to Lenarcik:

\begin{lema}(\cite[Proposition 61]{Lenarcik})
\label{Lenarcik1}
Let $f$ be a reduced power series of order bigger than one. Let $A_{1}$ be the area limited by $\mathcal N_f$  and the lines $x-1=0$ and $y-1=0$. If $(m_{1},1), (1,n_{1})\in  \mathcal N_f$ then $\mu(\mathcal N_f)
=2 A_{1}+m_{1}+n_{1}-1.$
\end{lema}

\begin{lema}
\label{lema:A}
Let $A$ be the area between the Newton polygon of $f=x^{d_{1}}y^{d_{2}}g\in K[[x,y]]$ and the lines $x-d_1=0$ and $y-d_2=0$. Let $m=\ord f(x,0)$, $n=\ord f(0,y)$ (by convention $\ord 0=+\infty$). Then

\begin{equation}
 A= \left\{\begin{array} {ll}
\, A_1+\frac{m+m_1-1}{2}+\frac{n+n_1-1}{2},\;\; \vert \mathcal N_f \vert_1=m,\;\; \vert \mathcal N_f \vert_2=n
& \hbox{\rm if }  (d_{1},d_{2})=(0,0)\;\\
\, A_1+\frac{m_1+m-2}{2},\;\; \vert \mathcal N_f \vert_1=m-1,\;\; \vert \mathcal N_f \vert_2=n_1&\hbox{\rm if }  (d_{1},d_{2})=(1,0)\\
\, A_1+\frac{n_1+n-2}{2},\;\; \vert \mathcal N_f \vert_1=m_1,\;\; \vert \mathcal N_f \vert_2=n-1&\hbox{\rm if }  (d_{1},d_{2})=(0,1)\;\\
\, A_1,\;\; \vert \mathcal N_f \vert_1=m_1-1,\;\; \vert \mathcal N_f \vert_2=n_1-1 &\hbox{\rm if }  (d_{1},d_{2})=(1,1).\;\\
\end{array}
\right.
\end{equation}
\end{lema}
\noindent \begin{proof}
It is a consequence of Lemma \ref{Lenarcik1}.
\end{proof}

\begin{lema}(\cite[p.146]{Walenska})
\label{Lenarcik}
Let $f=x^{d_{1}}y^{d_{2}}g\in K[[x,y]]$ be a reduced power series with $g(0,0)=0$. Then

\begin{equation}\label{no-conveniente}
\mu(\mathcal N_f)= \left\{\begin{array} {ll}
\, [\mathcal N_f,\mathcal N_f]-\vert \mathcal N_f \vert_1-\vert \mathcal N_f \vert_2+1 &\hbox{\rm if }  (d_{1},d_{2})=(0,0)\;\\
\, [\mathcal N_f,\mathcal N_f]-\vert \mathcal N_f \vert_1+\vert \mathcal N_f \vert_2 &\hbox{\rm if }  (d_{1},d_{2})=(1,0)\\
\, [\mathcal N_f,\mathcal N_f]+\vert \mathcal N_f \vert_1-\vert \mathcal N_f \vert_2 &\hbox{\rm if }  (d_{1},d_{2})=(0,1)\;\\
\, [\mathcal N_f,\mathcal N_f]+\vert \mathcal N_f \vert_1+\vert \mathcal N_f \vert_2+1 &\hbox{\rm if }  (d_{1},d_{2})=(1,1).\;\\
\end{array}
\right.
\end{equation}

\end{lema}

\noindent \begin{proof} We have $[\mathcal N_f,\mathcal N_f]=2A$ and by Lemma \ref{Lenarcik1}, $\mu(\mathcal N_f)=2A_1+m_1+n_1-1$. Use Lemma \ref{lema:A}.
\end{proof}

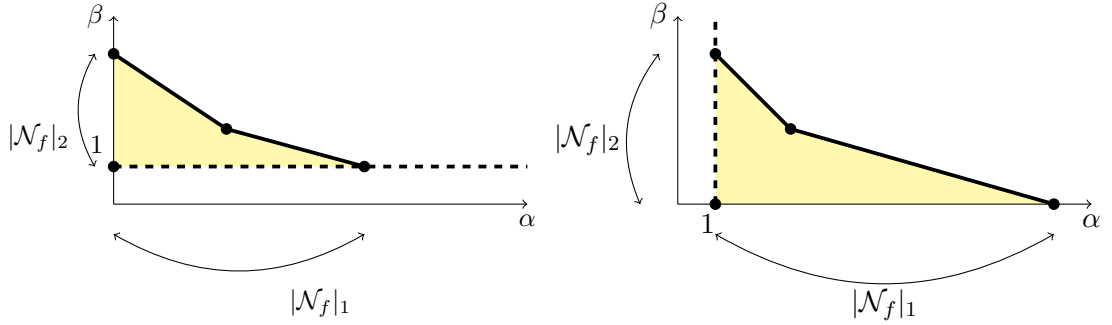
\begin{figure}[h!] 
\begin{center}
\begin{tikzpicture}[x=0.5cm,y=0.5cm] 
\begin{scope}[shift={(0,0)}]
\tikzstyle{every node}=[font=\small]
\fill[fill=yellow!40!white] (0,4) --(0,1)-- (20/3,1)--(3,2)-- cycle;
\node[draw,circle,inner sep=1.4pt,fill, color=black] at (0,4){};
\node[draw,circle,inner sep=1.4pt,fill, color=black] at (3,2){};
\node[draw,circle,inner sep=1.4pt,fill, color=black] at (0,1){};
\node[draw,circle,inner sep=1.4pt,fill, color=black] at (20/3,1){};
\node [left] at (0,1.6) {$1$};
\draw[<->] (-0.5,4) to [bend right] (-0.5,1);
\draw[->] (0,0) -- (11,0) node[right,below] {$\alpha$};
\draw[->] (0,0) -- (0,5) node[above,left] {$\beta$};
\draw[-, line width=0.5mm] (0,4) -- (3,2);
\draw[-, line width=0.5mm] (3,2) -- (20/3,1);
\draw[-, dashed, line width=0.5mm] (0,1) -- (11,1);
\draw[<->] (0,-0.8) to [bend right] (20/3,-0.8);
\draw[-] (5.5,-2) node[below] {$\vert \mathcal N_{f}\vert_1$};
\draw[-] (-2,2.4) node[right,below] {$\vert \mathcal N_{f}\vert_2$};
\end{scope}


\begin{scope}[shift={(15,0)}]
\fill[fill=yellow!40!white] (1,4) --(1,0)-- (10,0)--(3,2)-- cycle;
\node[draw,circle,inner sep=1.4pt,fill, color=black] at (1,4){};
\node[draw,circle,inner sep=1.4pt,fill, color=black] at (10,0){};
\node[draw,circle,inner sep=1.4pt,fill, color=black] at (3,2){};
\node[draw,circle,inner sep=1.4pt,fill, color=black] at (1,0){};
\node [left, below] at (0.8,0) {$1$};
\draw[<->] (-0.5,4) to [bend right] (-1,0);
\draw[->] (0,0) -- (11,0) node[right,below] {$\alpha$};
\draw[->] (0,0) -- (0,5) node[above,left] {$\beta$};
\draw[-, line width=0.5mm] (1,4) -- (3,2);
\draw[-, line width=0.5mm] (3,2) -- (10,0);
\draw[-, dashed, line width=0.5mm] (1,0) -- (1,5);
\draw[<->] (1,-0.8) to [bend right] (10,-0.8);
\draw[-] (5.5,-2) node[below] {$\vert \mathcal N_{f}\vert_1$};
\draw[-] (-2.4,2.4) node[right,below] {$\vert \mathcal N_{f}\vert_2$};
\end{scope}
\end{tikzpicture}
\end{center}
\caption{$f(x,y)=xg(x,y)$ and $f(x,y)=yg(x,y)$}  
   \label{fig:non-degenerate}
    \end{figure}

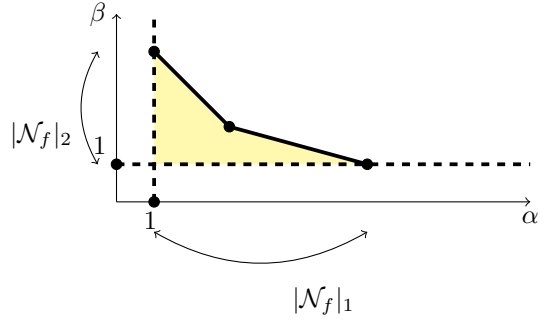
\begin{figure}[h!] 
\begin{center}
\begin{tikzpicture}[x=0.5cm,y=0.5cm] 
\fill[fill=yellow!40!white] (1,4) --(1,1)-- (20/3,1)--(3,2)-- cycle;
\tikzstyle{every node}=[font=\small]
\node[draw,circle,inner sep=1.4pt,fill, color=black] at (1,4){};
\node[draw,circle,inner sep=1.4pt,fill, color=black] at (3,2){};
\node[draw,circle,inner sep=1.4pt,fill, color=black] at (1,0){};
\node[draw,circle,inner sep=1.4pt,fill, color=black] at (0,1){};
\node[draw,circle,inner sep=1.4pt,fill, color=black] at (20/3,1){};
\node [left] at (0,1.5) {$1$};
\node [right, below] at (0.9,0) {$1$};
\draw[<->] (-0.5,4) to [bend right] (-0.5,1);
\draw[->] (0,0) -- (11,0) node[right,below] {$\alpha$};
\draw[->] (0,0) -- (0,5) node[above,left] {$\beta$};
\draw[-, line width=0.5mm] (1,4) -- (3,2);
\draw[-, line width=0.5mm] (3,2) -- (20/3,1);
\draw[-, dashed, line width=0.5mm] (0,1) -- (11,1);
\draw[-, dashed, line width=0.5mm] (1,0) -- (1,5);
\draw[<->] (1,-0.8) to [bend right] (20/3,-0.8);
\draw[-] (5.5,-2) node[below] {$\vert \mathcal N_{f}\vert_1$};
\draw[-] (-2,2.4) node[right,below] {$\vert \mathcal N_{f}\vert_2$};
\end{tikzpicture}
\end{center}
\caption{$f(x,y)=xyg(x,y)$}  
   \label{fig:non-degenerate}
    \end{figure}

A power series $f$ will be called {\it elementary} if $f$ is convenient and $\mathcal N_f$ contains only one edge $S$. The pair $(m,n)=(\vert S \vert_{1},\vert S \vert_{2})=(\ord f(x,0), \ord f(0,y))$ is by definition the {\it bidegree} of $f$  and we will denote it by $\bideg(f)$. In what follows we write $\Ini f=\ini (f,S)$. After \cite[Chapter 2]{Pfister-Ploski}, every convenient irreducible power series is elementary. If $f,g$ are elementary of bidegree $(m,n)$ resp. $(m',n')$ such that $\dfrac{m}{n}=\dfrac{m'}{n'}$, then $fg$ is elementary of bidegree $(m+m',n+n')$. Moreover, $\Ini fg=\Ini f \cdot \Ini \,g$.

\begin{lema}\label{homo}
If $f\in K[[x,y]]$ is elementary of bidegree $(m,n)$ and $d=\gcd(m,n)$ then $\Ini f(x,y)=F(x^{m/d},y^{n/d})$ where $F=F(u,v)$ is a homogeneous form of degree $d$. Moreover, if $f$ is irreducible then $\Ini f(x,y)=\left( ax^{m/d}+b y^{n/d}\right)^{d}.$
\end{lema}

\noindent \begin{proof}
The polynomial $\Ini f$ is a linear combination of monomials $x^{\alpha}y^{\beta}$, where $\alpha n+\beta m=nm$. It is easy to check that $\alpha=\frac{m}{d}\alpha_1$, $\beta=\frac{m}{d}\beta_1$ for some 
$\alpha_1, \beta_1\in \mathbb N$.
Moreover, $\alpha_1+\beta_1=d$. Therefore $x^{\alpha}y^{\beta}=x^{\frac{m}{d}\alpha_1}y^{\frac{n}{d}\beta_1}$ and $\Ini f=F(x^{\frac{m}{d}}y^{\frac{n}{d}}),$ where $F=F(u,v)$ is a homogeneous polynomial  of degree $d$. Let $F(u,v)=\prod_{i=1}^s(a_iu+b_iv)^{e_i}$ where $a_ib_j\neq a_jb_i$ for $i\neq j$. Then $\Ini f=F(x^{\frac{m}{d}}y^{\frac{n}{d}})=\prod_{i=1}^s(a_ix^{\frac{m}{d}}+b_iy^{\frac{n}{d}})^{e_i}.$ By Hensel's lemma (see \cite[Lemma A.1] {A-L-M}, \cite{Kunz}) we get $f(x,y)=g_1(x,y)\cdots g_s(x,y)\in K[[x,y]]$ where $\Ini g_i=(a_ix^{\frac{m}{d}}+b_iy^{\frac{n}{d}})^{e_i}$ for $i\in \{1, \ldots, s\}$. If $f$ is irreducible then $s=1$, $d=e_1$, $f=g_1$ and $\Ini f=\Ini \,g_1=(a_1x^{\frac{m}{d}}+b_1y^{\frac{n}{d}})^{d}$.
\end{proof}

\begin{coro}
\label{bi}
If $f$ is non-degenerate, convenient and irreducible power series of bidegree $(m,n)$  then $\gcd(n,m)=1$.
\end{coro}

\begin{lema}[Newton factorization] \label{Newtonfact}
Let $f\in K[[x,y]]$ be convenient. Then $f=\prod_{S\in \mathcal N_f}f_{S}$ in $K[[x,y]]$ where $f_{S}$ are elementary. Moreover,  the bidegree of $f_{S}$ is $(\vert S\vert_{1},\vert S\vert_{2})$ and $\Ini f_{S}=c \cdot \overline \ini (f,S)$, for some $c\in K\backslash \{0\}$.
\end{lema}
\noindent \begin{proof}
Firstly we prove that any convenient power series  is a product of elementary power series. If $f$ is elementary of bidegree $(m,n)$ then we put $I(f)=\frac{m}{n}$. Let $f=f_1\cdots f_r$ be the factorization into irreducible factors of a convenient power series. Let $\{I(f_i)\;:\;1\leq i\leq r\}=\{\omega_j\;:\;1\leq j\leq s\}$ where $\omega_1<\omega_2<\cdots <\omega_s$. For any $j\in \{1,\ldots, s\}$ we put $A_j:=\{k\in \{1,\ldots, s\}\;:\;I(f_k)=\omega_j\}$ and $g_j:=\prod_{i\in A_j}f_i$. Then $g_j$ is an elementary power series and $f=g_1\cdots g_s$ with  $I(g_i)<I(g_{j)}$  for any $i\neq j$. Let $\bideg(g_{k})=(m_{k},n_{k})$. Since 
$\frac{m_{1}}{n_{1}}<\cdots <\frac{m_{s}}{n_{s}}$ the points $v_{k}=\left(\sum_{i=1}^{k}m_{i},\sum_{i=k+1}^{s}n_{i}\right)$ with $k\in\{1, \ldots, s\}$ (by convention the empty sum equals zero) are vertices of $\mathcal N_f$. Let $S^{(k)}$ be the segment of $\mathcal N_f$ with vertices $v_{k-1}$ and $v_{k}$ for $k\in \{1,\ldots,s\}$, so $(\vert S^{(k)} \vert_{1},\vert S^{(k)} \vert_{2})=(m_{k},n_{k})$.
If $S\in \mathcal N_f$ then $S=S^{(k)}$ for some $k\in \{1,\ldots,s\}$ and we put $f_{S}=g_{k}$. Therefore 
$f=\prod_{S\in \mathcal N_f}f_{S}$ where $f_{S}$ are elementary, $\bideg(f_{S})=(\vert S \vert_{1},\vert S \vert_{2})$ and $\Ini f_{S}=c\cdot \overline \ini (f,S)$ for some $c\in K\backslash \{0\}$.  
\end{proof}

\begin{coro}
If $f\in K[[x,y]]$ is non-degenerate then $f_{S}$ are non-degenerate for any $S\in \mathcal N_f.$
\end{coro}

\noindent For any two power series  $f,g\in K[[x,y]]$ we put $i_0(f,g):=\dim_{K} K[[x,y]]/(f,g)$ and call it the \emph{intersection multiplicity} of $f$ and $g$.

\begin{lema}\label{intmul}
If $ \mathcal N_f=\{S\}$ and $ \mathcal N_g=\{T\}$ are elementary then
$i_{0}(f,g)\geq [S,T]$ with equality if and only if $S$ and $T$ are not parallel or the system of equations $\Ini f=0$, $\Ini g=0$ has the unique solution $(x,y)=(0,0)$.
\end{lema}
\noindent \begin{proof}
Put $\bideg(f):=(m,n)$ and $\bideg(g):=(m_{1},n_{1})$. We have to check that $i_{0}(f,g)\geq \min \{mn_{1},m_{1}n\}$ with equality if and only if $\frac{m}{n}=\frac{m_{1}}{n_{1}}$ or the system of equations
$\Ini f=0$, $\Ini g=0$ has the only solution $(x,y)=(0,0)$. Put $f(x,y)=\sum_{ij}a_{ij}x^{i}y^{j}$. Let $\overrightarrow w=(n,m)\in \mathbb N_{+ }^{2}$. Then
$\ord_{\overrightarrow w}(f):=\inf \{ni+jm\;:\;a_{ij}\neq 0\}=nm$ and
$\ini_{\overrightarrow w} f:=\sum_{i n+j m=nm} a_{ij}x^{i}y^{j}=\Ini f$. Let us distinguish two cases.\\

\noindent {\it Case 1:} $\frac{m}{n}\neq \frac{m_{1}}{n_{1}}.$ We may assume $\frac{m}{n}<\frac{m_{1}}{n_{1}}.$ Then $\ord_{\overrightarrow w}(g)=mn_{1}=\min\{mn_{1}, m_{1}n\}$ and $\ini_{\overrightarrow w} g=cy^{n_{1}}$ for $c\neq 0$. Therefore the system of equations $\ini_{\overrightarrow w} f=0$ and $\ini_{\overrightarrow w} g=0$ has the unique solution $(x,y)=(0,0)$ and we get
\[
i_{0}(f,g)=\frac{\ord_{\overrightarrow w} f {\ord_{\overrightarrow w} g}}{mn}=\ord_{\overrightarrow w} g=mn_1=\min \{mn_{1},m_{1}n\},
\]
by \cite[Lemma A.1]{GB-P2018}.\\

\noindent {\it Case 2:} $\frac{m}{n}=\frac{m_{1}}{n_{1}}.$ We check $\ord_{\overrightarrow w}(g)=mn_{1}$ and $\ini_{\overrightarrow w} g=\Ini g$. Again by \cite[Lemma A.1]{GB-P2018} we get $i_{0}(f,g)\geq \ord_{\overrightarrow w} g=m n_{1}=nm_{1}$ with equality if the system $\Ini f=0$, $\Ini g=0$ has the unique solution $(x,y)=(0,0)$.
\end{proof}

\section{Main result}
\label{sec:main}

The following theorem is the main result of this note:

\begin{teorema}\label{main}
Let $f\in K[[x,y]]$ be a reduced power series. Then 
\begin{enumerate}
\item $\bar \mu (f)- \mu(\mathcal N_f)\geq r(\mathcal N_f)-r(f)\geq 0$.
\item If $f$ is non-degenerate then $\bar \mu (f)= \mu(\mathcal N_f)$ and $r(\mathcal N_f)=r(f)$.
\end{enumerate}
\end{teorema}

The first statement of Theorem \ref{main} was proved in \cite[Proposition 7]{BGreuel}. We  provide a new and simple proof of it.
The proof of Theorem \ref{main} is given in Section \ref{proof}.\\

As an immediate consequence of Theorem \ref{main} we have

\begin{coro}(\cite[Lemma 4]{BGreuel})
Let $f\in K[[x,y]]$ be a reduced power series. We have $r(f)\leq r(\mathcal N_f)$ and
if $f$ is non-degenerate then $r(f)=r(\mathcal N_f)$.
\end{coro}

\begin{coro}(\cite[Proposition 3.17]{Be-Pe}, \cite[Proposition 5]{BGreuel})
Let $f\in K[[x,y]]$. We have $\delta(\mathcal N_f)\leq \delta(f)$ and if 
$f$ is non-degenerate then $\delta(\mathcal N_f)=\delta(f)$. 
\end{coro}
\begin{proof} From the definition of the invariant Milnor number of $f$ and the equality \eqref{eq:delta} we have $\bar\mu(f)-\mu(\mathcal N_f)=2(\delta(f)-\delta(\mathcal N_f))+r(\mathcal N_f)-r(f)$. We use Theorem \ref{main}.
\end{proof}
\begin{coro}(\cite[Theorem 9]{BGreuel})
Let $f\in K[[x,y]]$ be a reduced power series. If $f$ is strongly non-degenerate then $f$ is tame, i.e., $\mu(f)=\bar \mu (f)$.
\end{coro}
\begin{proof}
By Kouchnirenko's planar theorem  \cite[Proposition 4]{BGreuel} we have $\mu(f)=\mu(\mathcal N_f)$. On the other hand by Theorem \ref{main} we get $\bar \mu(f)=\mu(\mathcal N_f)$. Therefore $\mu(f)=\bar \mu(f)$.
\end{proof}

\section{Proof of the main result}\label{proof}
We begin with the proof of Theorem \ref{main} for convenient power series. Firstly we consider the case of elementary power series.  Let $f\in K[[x,y]]$ be an elementary power series of bidegree $(m,n)$. Let $d:=\gcd(m,n)$. Then the theorem reduces to the following statement:
\begin{equation}\label{thirred}
\bar \mu (f)- (n-1)(m-1)\geq d-r(f)\geq 0.
\end{equation}
If $f$ is non-degenerate then $\bar \mu (f)=(n-1)(m-1)$ and $r(f)=d$.\\

We distinguish two cases. Suppose first that $f$ is irreducible, that is $r(f)=1$. 

\begin{lema}
\label{cota}
Let $f\in K[[x,y]]$ be irreducible with semigroup of values $\Gamma(f)=\langle v_0,v_1,\ldots,v_h\rangle$. If $c$ is the conductor of $\Gamma(f)$ then $c\geq (v_0-1)(v_1-1)+\gcd(v_0,v_1)-1$. The equality $c=(v_0-1)(v_1-1)$ holds if and only if $\gcd(v_0,v_1)=1$.
\end{lema}

\noindent \begin{proof}
Let us define Puiseux characteristic sequence $b_0, b_1,\ldots, b_h$ by putting $b_0=v_0$, $b_k=v_k-\sum_{i=1}^h(n_i-1)v_i$ for $k\in \{1,\ldots,h\}.$ Note that $\gcd(b_0,\ldots,b_k)=e_k$ for $k\in \{0,\ldots,h\}$ and $b_0<b_1<\cdots<b_h$. Moreover $c=\sum_{k=1}^h(e_{k-1}-e_k)(b_k-1)$ (see for example \cite[Chapter 3, p. 58]{Pfister-Ploski}. If $e_1=1$ then $c=(e_0-e_1)(b_1-1)=(b_0-1)(b_1-1)$. Therefore we may assume that $h>1$. We have
\begin{eqnarray*}
c&=&(e_0-e_1)(b_1-1)+\sum_{k=2}^h(e_{k-1}-e_k)(b_k-1)\\
&\geq & (e_0-e_1)(b_1-1)+\sum_{k=2}^h(e_{k-1}-e_k)(b_2-1)\\
&=&(e_0-e_1)(b_1-1)+(e_1-1)(b_2-1)\\
&=& (e_0-e_1)(b_1-1)+(e_1-1)(b_2-b_{1}+b_{1}-1)\\
&=& (e_0-1)(b_1-1)+(e_1-1)(b_2-b_{1})\\
&\geq& (b_0-1)(b_1-1)+e_1-1,\;\;\;\hbox{\rm since } b_2-b_1\geq 1.\\
\end{eqnarray*}
\end{proof}

Suppose that $r(f)=1$. Since $\bar \mu (f)=c(f)=c$ we have, by Lemma \ref{cota}, $\bar \mu (f)\geq (v_0-1)(v_1-1)+\gcd(v_0,v_1)-1$. The power series $f$ being unitangent we have $m=\ord f(0,y)=\ord f$ or $n=\ord f(x,0)=\ord f$. Assume that $m=\ord f$. Then $m<n\leq v_{1}$ (see \cite{GB-P2019}).
If the axis $y=0$ has maximal contact with the curve $f(x,y)=0$ then $n=v_1$ and by Lemma \ref{cota} we get  
\[
\bar \mu (f)\geq (v_{0}-1)(v_{1}-1)-\gcd(v_{0},v_{1})+1=(m-1)(n-1)-d+1\geq 0.
\]

If $n<v_{1}$ then $n\equiv 0$ (mod $m$), $d=\gcd(m,n)=m$ and we get

\begin{eqnarray*}
\bar \mu(f)&\geq &(v_0-1)(v_1-1)=(m-1)(v_{1}-n+n-1)\\
&=&(m-1)(n-1)+(v_{1}-1)(m-1)\\
&\geq &(m-1)(n-1)+m-1=
(m-1)(n-1)+d-1.
\end{eqnarray*}

\noindent Suppose that $f$ is non-degenerate. Then, by Corollary \ref{bi}, $d=\gcd(n,m)=1$. Consequently, by Lemma \ref{cota}, $\bar \mu(f)=(m-1)(n-1)+d-1$.

Suppose now that $f$ is elementary but $r(f)>1$. Recall that any irreducible convenient power series is elementary.

\begin{lema}\label{lemma:elementary}
Let $f$ be an elementary power series with $\bideg(f)=(m,n)$ and $f=f_1\cdots f_r$ the factorization of $f$ into irreducible factors with  $\bideg(f_i)=(m_i,n_i)$. If $d=\gcd(m,n)$ and $d_i=\gcd(m_i,n_i)$ then
\begin{enumerate}
\item $\frac{m_i}{d_i}=\frac{m}{d}$ and $\frac{n_i}{d_i}=\frac{n}{d}$ for any $i\in \{1,\ldots,r\}$.
\item $\sum_{i=1}^r d_i=d$.
\end{enumerate}
Moreover, $r\leq d$ with equality if $f$ is non-degenerate.
\end{lema}
\noindent \begin{proof}
By Lemma \ref{homo} we have $\Ini f(x,y)=F(x^{m/d},y^{n/d})$ for some homogeneous polynomial $F$ of degree $d$.
Since $f_i$ are elementary $\Ini f(x,y)=\Ini f_1(x,y) \cdots \Ini f_r(x,y)$. By Lemma \ref{homo} $\Ini f_i(x,y)=(a_ix^{\frac{m_i}{d_i}}+b_iy^{\frac{n_i}{d_i}})^{d_i}$ for some $a_i,b_i\in K$. Then $a_ix^{\frac{m_i}{d_i}}+b_iy^{\frac{n_i}{d_i}}$ is an irreducible factor of $F(x^{m/d},y^{n/d})$, which implies $\frac{m_i}{d_i}=\frac{m}{d}$ and $\frac{n_i}{d_i}=\frac{n}{d}$ for any $i\in \{1,\ldots,r\}$. Since $f(x,0)=\prod_{i=1}^r f_i(x,0)$ we have $m=\ord f(x,0)=\sum_{i=1}^r \ord f_i(x,0)=\sum_{i=1}^r m_i=\sum_{i=1}^r d_i \frac{m}{d}$ whence $\sum_{i=1}^r d_i=d$. Obviously $r\leq d$. If $f$ is non-degenerate then $f_i$ are non-degenerate and $d_i=1$ for $i\in \{1,\ldots,r\}$ by Corollary \ref{bi}.
Therefore $r=d$.
\end{proof}

\medskip

By the third statement of Proposition \ref{prop:1} we get
$\bar \mu (f)+r-1=\sum_{i=1}^{r}\bar \mu (f_{i})+2\sum_{1\leq i<j\leq r}i_{0}(f_{i},f_{j})$. By the irreducible elementary case we have $\bar \mu (f_i)\geq \left(\dfrac{m}{d}d_i-1\right) \left(\dfrac{n}{d}d_i-1\right)+(d_i-1).$ Moreover, by Lemma \ref{intmul},  $i_0(f_i,f_j)\geq \dfrac{mn}{d^2}d_{i}d_{j}$. Therefore we get

\begin{eqnarray*}
\bar \mu(f)+r-1&\geq& \sum_{i=1}^{r}\left[\left( \frac{m}{d}d_{i}-1\right)  \left( \frac{n}{d}d_{i}-1\right) +(d_{i}-1) \right] +2 \sum_{1\leq i<j\leq r}\frac{mn}{d^{2}}d_{i}d_{j}\\
&=&\frac{mn}{d^2}\left(\sum_{i=1}^r d_i^2+2\sum_{1\leq i<j\leq r}d_id_j\right)+\left(\frac{-n-m}{d}+1\right)\sum_{i=1}^r d_i\\
&=&mn-n-m+d.
\end{eqnarray*}

Whence $\bar \mu(f)+r-1\geq(n-1)(m-1)+d-1$ which implies the inequality \eqref{thirred}. If $f$ is non-degenerate then $d_i=1$ for $i\in\{1,\ldots,r\}$, $\bar \mu(f_i)=\left(\frac{m}{d}-1\right)\left(\frac{n}{d}-1\right)$ and $i_0(f_i,f_j)=\frac{mn}{d^2}$ and the inequalities become equalities. Moreover, $r(f)=r=d$ by Lemma \ref{lemma:elementary}.\\

Let us prove now the general case, that is $\sigma:=\sharp \mathcal N_f>1$. Let $f=\prod_{S\in \mathcal N_f}f_S$ be the Newton factorization  of $f$. By the third statement of Proposition \ref{prop:1} we get
\[
\bar \mu(f)+\sigma-1=\sum_{S\in \mathcal N_f}\bar \mu(f_S)+\sum_{S\neq T}i_0(f_S,f_T)=\sum_{S\in \mathcal N_f}\bar \mu(f_S)+\sum_{S\neq T}[S,T],
\]
where $S$ and $T$ are not parallel. Since $f_S$ is elementary of bidegree $(\vert S\vert_1,\vert S\vert_2)$ we get
\[
\bar \mu(f_S)\geq (\vert S\vert_1-1)(\vert S\vert_2-1)+\gcd(\vert S\vert_1,\vert S\vert_2)-r(f_S).
\]

A simple calculation shows that
\[
\bar \mu(f)+\sigma-1\geq [\mathcal N_f,\mathcal N_f]-\vert \mathcal N_f\vert_1-\vert \mathcal N_f\vert_2+\sigma+r(\mathcal N_f)-r(f).
\]

Therefore $\bar \mu(f)\geq \mu(\mathcal N_f)+r(\mathcal N_f)-r(f)$. If $f$ is non-degenerate then $f_S$ is non-degenerate. Thus $\bar \mu(f_S)=\mu(\mathcal N_{f_S})$ and $r(f_S)=r(\mathcal N_{f_S})=\gcd(\vert S\vert_1,\vert S\vert_2)$. Using the Newton factorization we get $\overline \mu(f)=\mu(\mathcal N_{f})$. Obviously $r(f)=\sum_S r(f_S)=\sum_S \gcd (\vert S\vert_1,\vert S\vert_2)=r(\mathcal N_{f})$.\\

It remains to prove Theorem \ref{main} for non-convenient power series. \\

Let $f(x,y)=x^{d_{1}}y^{d_{2}}g(x,y)$ where $g=g(x,y)$ is a convenient reduced power series or a unit. We assume that $g(0,0)=0$ (if $g(0,0)\neq 0$  then $\overline \mu(f)=\mu(\mathcal N_f)$ and $r(f)=\mu(\mathcal N_f)$).\\

Because the length of a segment is the same on parallel axes we have

\begin{equation}
\vert \mathcal N_{f}\vert_i=\vert \mathcal N_{g}\vert_i \;\;\hbox{\rm for $i=1,2$, $[\mathcal N_{f},\mathcal N_{f}]=
[\mathcal N_{g},\mathcal N_{g}]$ and $r(\mathcal N_{f})-r(f)=r(\mathcal N_{g})-r(g)$}.
\end{equation}

Since we have already proved Theorem \ref{main} for convenient power series we get

\begin{equation}
\label{Th2:eq1}
\overline \mu (g)-\mu(\mathcal N_g)\geq r(\mathcal N_g)-r(g)\geq 0,
\end{equation}

\noindent and the equalities $\overline \mu (g)=\mu(\mathcal N_g)$ and $r(\mathcal N_g)=r(g)$ holding for non-degenerate $g$.

By Proposition \ref{prop:1}   we get
\begin{eqnarray*}
\bar \mu(f)+2&=&\bar \mu (x)+\bar \mu (y)+\bar \mu(g)+2i_0(g,x)+2i_0(g,y)+2i_0(x,y)\\
&=&\bar \mu(g)+2\ord g(0,y)+2\ord g(x,0)+2,
\end{eqnarray*}

and 
\begin{eqnarray*}
\bar \mu(f)&=&\bar \mu(g)+2\vert \mathcal N_g \vert_{1}+2\vert \mathcal N_g \vert_{2}\\
&\geq &\mu(\mathcal N_{g})+r(\mathcal N_{g})-r(g)+2\vert \mathcal N_g \vert_{1}+2\vert \mathcal N_g \vert_{2}\\
&=&[\mathcal N_g,\mathcal N_g]+\vert \mathcal N_g \vert_1+\vert \mathcal N_g\vert_2+1+r(\mathcal N_{g})-r(g)\\
&=&[\mathcal N_f,\mathcal N_f]+\vert \mathcal N_f \vert_1+\vert \mathcal N_f \vert_2+1+r(\mathcal N_{f})-r(f)\\
&\geq &\mu(\mathcal N_{f})+r(\mathcal N_{f})-r(f).
\end{eqnarray*}

If $f$ is non-degenerate then $g$ is non-degenerate and we get $\bar \mu(f)= \mu(\mathcal N_{f})$ and $r(\mathcal N_{f})=r(f)$.

\medskip
\noindent
{\small Evelia Rosa Garc\'{\i}a Barroso\\
Departamento de Matem\'aticas, Estad\'{\i}stica e I.O. \\
Secci\'on de Matem\'aticas, Universidad de La Laguna\\
Apartado de Correos 456\\
38200 La Laguna, Tenerife, Espa\~na\\
e-mail: ergarcia@ull.es}

\medskip

\noindent {\small Arkadiusz P\l oski\\
Department of Mathematics and Physics\\
Kielce University of Technology\\
Al. 1000 L PP7\\
25-314 Kielce, Poland\\
e-mail: matap@tu.kielce.pl}


\begin{thebibliography}{GGGGG}


\bibitem[A-L-M] {A-L-M} Artal Bartolo, E.; Luengo, I.; Melle-Hern‡ndez, A.: High-school algebra of the theory of dicritical divisors: atypical fibers for special pencils and polynomials. J. Algebra Appl. 14 (2015), no. 9, 1540009, 26 pp.

\bibitem[Be-Pe]{Be-Pe} Beelen, P.; Pellikaan, R.: The Newton polygon of plane curves with many rational points. Special issue dedicated to Dr. Jaap Seidel on the occasion of his 80th birthday (Oisterwijk, 1999). Des. Codes Cryptogr. 21 (2000), no. 1-3, 41-67. 

\bibitem[Bo-G-M] {BGreuel} Boubakri, Y., Greuel G-M, Markwig, T.: Invariants of hypersurface singularities in  positive characteristic. Rev. Mat. Complut. {\bf 25} (2010), 61-85. 


\bibitem[D]{Deligne} Deligne, P. La formule de Milnor, Sem. Geom. alg\'ebrique, Bois-Marie 1967-1969, SGA 7 II, Lect. Notes Math., 340, Expos\'e XVI, 197-211 (1973).

\bibitem[GB-P\l1]{GB-P2015}Garc\'{\i}a Barroso, E.; P\l oski, A.:   {\em An approach to plane algebroid branches. } Revista Matem\'atica Complutense, 28 (1) (2015), 227-252 . 

\bibitem[GB-P\l2]{GB-P2018} Garc\'{\i}a Barroso, E.; P\l oski, A.:   {\em On the Milnor formula in arbitrary characteristic.} In Singularities, Algebraic Geometry, Commutative Algebra and Related Topics. Festschrift for Antonio Campillo on the Occasion of his 65th Birthday. G.M. Greuel, L. Narva\'ez and S. Xamb\'o-Descamps eds. Springer, (2018), 119-133. doi: 10.1007/978-3-319-96827-8\_5.

\bibitem[GB-P\l3]{GB-P2019} Garc\'{\i}a Barroso, E.; P\l oski, A.:   {\em Contact exponent and Milnor number of plane curve singularities.} In Analytic and Algebraic Geometry, 3. T. Krasi\'nski, Stanis\l aw Spodzieja (Eds.) \L \'od\'z University Press (2019), 93-109.

\bibitem[G-N]{G-N} Greuel, G-M, Nguyen, H.D.: Some remarks on the planar Kouchnirenko's theorem. Rev. Mat. Complut {\bf 25}  (2012), 557-579.

\bibitem[Kou]{Kouchnirenko} Kouchnirenko, A. G.: Poly\`edres de Newton et nombres de Milnor. 
Invent. Math. 32 (1976), no. 1, 1-31.

\bibitem[Ku]{Kunz} Kunz, E.:  Introduction to Plane Algebraic Curves. Translated from the 1991 German edition by Richard G. Belshoff. Birkh\" auser, Boston (2005)

\bibitem[MH-W]{Melle-Wall} Melle-Hern\'andez, A., Wall C.T.C.: Pencils of cuves on smooth surfaces. Proc. Lond. Math. Soc., {\bf 83}(2) (2001), 257-278.

\bibitem[Le]{Lenarcik} Lenarcik, A.: On the Jacobian Newton polygon of plane curve singularities. Manuscripta Math. 125 (2008), 309-324

\bibitem[Pf-P\l]{Pfister-Ploski} Pfister, G., P\l oski, A.: Plane algebroid curves in arbitrary characteristic. IMPAN Lecture Notes (in print).

\bibitem[P\l]{Ploski} P\l oski, A.: Milnor number of a plane curve and Newton polygons. 
Effective methods in algebraic and analytic geometry (Bielsko-Bia\l a, 1997).
Univ. Iagel. Acta Math. No. 37 (1999), 75-80. 

\bibitem[W]{Walenska} Walenska, J.: Jumps of Milnor numbers in families of non-degenerate and non-convenient singularities. In Analytic and Algebraic Geometry. T. Krasi\'nski, Stanis\l aw Spodzieja (Eds.) \L \'od\'z University Press (2013), 141-153.

\end{thebibliography}
\end{document}